\documentclass[11pt]{amsart}

\usepackage{amssymb}
\usepackage{microtype}
\usepackage{xurl}
\usepackage[hidelinks]{hyperref}
\newtheorem{theorem}{Theorem}[section]
\newtheorem{proposition}[theorem]{Proposition}
\newtheorem{lemma}[theorem]{Lemma}

\theoremstyle{definition}
\newtheorem{definition}[theorem]{Definition}
\theoremstyle{remark}
\newtheorem{remark}[theorem]{Remark}

\newcommand{\HH}{\mathbb H}
\newcommand{\Om}{\Omega}
\newcommand{\Hess}{\operatorname{Hess}}

\newcommand{\tr}{\operatorname{tr}}

\DeclareMathOperator{\artanh}{artanh}
\DeclareMathOperator{\dist}{dist}

\title[Log-Concavity and Level-Set Horoconvexity of the Eigenfunction]
{Log-Concavity and Level-Set Horoconvexity of the First Eigenfunction
on Horoconvex Domains in the Hyperbolic Plane}

\author{Xianzhe Dai}
\address{Department of Mathematics, University of California, Santa Barbara,
CA 93106, USA}
\email{dai@math.ucsb.edu}

\author{John M. Ennis}
\address{Aigora, Richmond, VA, USA}
\email{john.m.ennis@aigora.com}

\author{Xuan Hien Nguyen}
\address{Department of Mathematics, Iowa State University, Ames, IA 50011, USA}
\email{xhnguyen@iastate.edu}

\author{Guofang Wei}
\address{Department of Mathematics, University of California, Santa Barbara,
CA 93106, USA}
\email{wei@math.ucsb.edu}
\thanks{G. Wei is partially supported by NSF DMS 2403557.}

\date{August 5, 2026}

\subjclass[2020]{Primary 35B35, 35P15; Secondary 53C21, 58J50}

\keywords{First Dirichlet eigenfunction, log-concavity, hyperbolic plane,
horoconvexity, level sets, Killing field, nodal set}

\begin{document}
\begin{abstract}
Let \(\Om\subset\HH^2\) be a bounded smooth horoconvex domain and let
\(\psi_1>0\) be its first Dirichlet eigenfunction. We prove that
\[
  \Hess_{\HH^2}(-\log\psi_1)>0
\]
throughout \(\Om\), with no restriction on the diameter or the first
eigenvalue. The proof is by contradiction. A degenerate Hessian would yield a shifted translation Killing derivative with a singular interior zero. Then the boundary-zero theorem of Grossi and
Provenzano shows that the shifted Killing derivative has exactly two zeros on the boundary. A nodal-domain argument on the surface rules
this out. As an application we prove that every superlevel set of
\(\psi_1\) is horoconvex: every
level curve has geodesic curvature at least \(1\). The Hessian bound makes
the shifted construction available for Killing fields with nonvanishing
rotation part, and yields the pointwise inequality
\(|(\Hess u)^{-1}J\nabla u|\le1\) for \(u=-\log\psi_1\), where \(J\) is
rotation by \(\pi/2\); a boundary-zero count for translation fields with
arbitrary axis completes the argument.
\end{abstract}

\maketitle

\section{Introduction}

Let \(\Om\) be a bounded smooth domain in the hyperbolic plane
\(\HH^2\). Its first Dirichlet eigenfunction is the
positive solution, unique up to scaling, of
\begin{equation}\label{eq:eigenfunction}
  -\Delta\psi_1=\lambda_1\psi_1\quad\text{in }\Om,
  \qquad
  \psi_1=0\quad\text{on }\partial\Om.
\end{equation}
The function \(\psi_1\) is log-concave when
\(\Hess\log\psi_1\leq0\), or equivalently when
\(\Hess(-\log\psi_1)\geq0\).

Brascamp and Lieb established log-concavity of the first Dirichlet eigenfunction on convex Euclidean domains \cite{BrascampLieb}, and Lee and Wang proved the analogous result for convex domains on the sphere \cite{LeeWang}. In hyperbolic space, however, negative curvature leads to a different picture: Shih constructed a geodesically convex domain in $\mathbb H^2$ whose first eigenfunction fails to be log-concave \cite{Shih}. This motivates imposing a stronger geometric condition on the domain, namely horoconvexity.

\begin{definition}\label{def:horoconvex}
A \emph{horoball} in \(\HH^2\) is a closed sublevel set \(\{b\leq c\}\) of a
Busemann function \(b\); its boundary \(\{b=c\}\) is a \emph{horocycle}. A
bounded smooth domain \(\Om\subset\HH^2\) is \emph{horoconvex} if its closure
is an intersection of horoballs.
\end{definition}

Two reformulations are used below \cite{Currier, BorisenkoMiquel}. Every point \(x\) of the boundary of a
horoconvex domain carries a \emph{supporting horoball}: a horoball
\(H\supset\overline\Om\) with \(x\in\partial H\). For a smooth boundary,
horoconvexity is equivalent to the inward geodesic-curvature inequality
\(k_g\geq1\).

Horoconvexity is a more fruitful notion of convexity in hyperbolic space, as evidenced by the following results. 
Grossi and Provenzano proved that the first eigenfunction of every smooth
horoconvex domain in \(\HH^2\) has a unique nondegenerate critical point
\cite{GrossiProvenzano}.  Khan, Saha, and Tuerkoen obtained a small-diameter
concavity estimate with respect to a conformally changed connection, but not the original one
\cite{KhanSahaTuerkoen}. 
Wei and Xiao proved the stronger super
log-concavity inequality for sufficiently small horoconvex domains in
\(\HH^n\) \cite{WeiXiao}. Their diameter
restriction is necessary for super log-concavity, even for balls in higher
dimension. Log-concavity is weaker and remains compatible with the large-diameter regime. Although recent work establishes large-diameter fundamental-gap estimates, it does not show log-concavity of the first eigenfunction itself \cite{DaiEnnisNguyenWei,KhanTuerkoen}.

Our first main result removes the diameter restriction for log-concavity
in dimension two. To the best of our knowledge, it is the first log-concavity theorem for all bounded smooth horoconvex domains in
\(\HH^2\) without a diameter restriction.

\begin{theorem}\label{thm:main}
Let \(\Om\subset\HH^2\) be a bounded smooth horoconvex domain. If \(\psi_1\)
is its first Dirichlet eigenfunction, then
\[
  \Hess(-\log\psi_1)>0
  \qquad\text{throughout }\Om.
\]
\end{theorem}

Our second result concerns the level sets of \(\psi_1\). Log-concavity makes
\(\psi_1\) quasiconcave, so Theorem \ref{thm:main} already implies that
every superlevel set \(\{\psi_1>t\}\) is geodesically convex. The super
log-concavity inequality of Wei and Xiao is a tensor inequality whose
tangential part says precisely that the superlevel sets are horoconvex; as
recalled above, their theorem requires a diameter restriction, and the
restriction is necessary for the full tensor inequality. The tangential
conclusion does not need it.

\begin{theorem}\label{thm:level}
Let \(\Om\subset\HH^2\) be a bounded smooth horoconvex domain. If \(\psi_1\)
is its first Dirichlet eigenfunction, then for every
\(t\in(0,\max\psi_1)\) the level curve \(\{\psi_1=t\}\) is a smooth closed
curve of geodesic curvature at least \(1\), and every superlevel set
\(\{\psi_1>t\}\) is horoconvex.
\end{theorem}

The constant \(1\) cannot be improved: on a ball of radius \(r\) the level
curves are concentric circles of geodesic curvature \(\coth\rho\) with
\(\rho<r\), the infimum over the level curves is \(\coth r\), and
\(\coth r\to1\) as \(r\to\infty\).

\paragraph{\it Proof strategy for Theorem \ref{thm:main}}
Set \(u=-\log\psi_1\). The eigenfunction equation gives the following.
\[
  \tr\Hess u=\Delta u=\lambda_1+|\nabla u|^2>0.
\]
At a maximum of \(\psi_1\), the Hessian of \(u\) is nonnegative. It follows
that the theorem will be proved once we know that \(\Hess u\) is never
singular. 
Thus the question is
reduced to ruling out a degenerate point of \(\Hess u\). 

Suppose that \(\Hess u\) is singular at \(p\), and choose a unit kernel
vector \(v\). There is a translation (hyperbolic)  Killing field \(K\) with
\[
  K(p)=v,\qquad \nabla K|_p=0.
\]
The kernel condition makes \(\nabla(K(\psi_1))(p)=0\), but in general
\(K(\psi_1)(p)\neq0\). This is the obstruction to using the obvious Killing
derivative \(K(\psi_1)\). Here it's important to consider the new (shifted) function
\[
  z=-\psi_1\bigl(K(u)-K(u)(p)\bigr)
   =K(\psi_1)+K(u)(p)\psi_1,
\]
which then satisfies \(z(p)=0\) and \(\nabla z(p)=0\), namely a singular zero at $p$. At the same time, \(z\) still solves
the first-eigenvalue equation, and the added term vanishes on
\(\partial\Om\), so \(z\) has the same boundary trace as \(K(\psi_1)\).

This is where the geometric input enters. Grossi and Provenzano proved that
the boundary trace \(K(\psi_1)\) has exactly two zeros for every translation
field of this form
\cite[Proposition 3.15(iii)]{GrossiProvenzano}. A first-eigenvalue solution
with two boundary zeros has only two nodal domains, one positive and one
negative. A singular interior zero, however, has at least four alternating
sectors. In a disk, two sectors of one sign can be joined to form a Jordan
curve that separates two sectors of the other sign. This contradicts the
connectedness of that nodal domain.
\newline

\paragraph{\it Proof strategy for Theorem \ref{thm:level}.}
It is crucial to consider a second (much larger) family of Killing fields. 
The translation fields in \cite{GrossiProvenzano} always have axis passing 
through the domain while we have to make use of those having axes in general position, 
possibly disjoint from \(\overline\Om\).
Section \ref{sec:count} proves the count for
arbitrary axes; the novel geometric ingredient is that a complete curve of constant
geodesic curvature less than \(1\) meets \(\partial\Om\) in at most two
points (Lemma \ref{lem:two-points}).

Write \(J\) for
rotation by \(\pi/2\) in the tangent planes. Every Killing field \(K\) on
\(\HH^2\) satisfies \(\nabla_XK=\omega JX\) for a scalar function \(\omega\),
and for \(p\in\Om\) the shifted derivative
\(z=K(\psi_1)+K(u)(p)\,\psi_1\) has a singular zero at \(p\) exactly when
\[
  \Hess u|_p\,K(p)=\omega(p)\,J\nabla u(p).
\]
Theorem \ref{thm:main} makes \(\Hess u\)
invertible, so the equation can instead be solved at every regular point,
with \(\omega(p)=1\) and \(K(p)=(\Hess u)^{-1}J\nabla u(p)\). The quantity
\(|K|^2-\omega^2\) is constant along every Killing field and determines its
type, and this \(K\) is a translation field precisely when
\(|(\Hess u)^{-1}J\nabla u|(p)>1\). For a translation field, the
boundary-zero count and the nodal argument forbid the singular zero. The
conclusion is the pointwise inequality
\begin{equation}\label{eq:star}
  \bigl|(\Hess u)^{-1}J\nabla u\bigr|\;\le\;1
  \qquad\text{on }\Om,
\end{equation}
and a two-by-two computation converts \eqref{eq:star} into the curvature
bound of Theorem \ref{thm:level}.

Sections 2 and 3 construct \(z\) and verify its boundary behavior. Section 4
proves the nodal statement, and Section 5 assembles these facts and proves Theorem~\ref{thm:main}.
Section \ref{sec:killing2} develops the structure of general Killing fields,
Section \ref{sec:count} proves the boundary-zero count for arbitrary axes,
and Section \ref{sec:level} completes the proof of Theorem \ref{thm:level}.

\begin{remark}
The proofs use two facts special to surfaces: a singular nodal point has
alternating sectors that can be separated by a Jordan curve, and a positive
determinant together with a positive trace settles definiteness for a
two-by-two Hessian. The Grossi--Provenzano boundary count is also
two-dimensional, as are the rotation scalar \(\omega\) and the tangency
count of Section \ref{sec:count}. We make no claim for \(\HH^n\),
\(n\geq3\), for either theorem.
\end{remark}

\section{Translation Killing fields and boundary zeros}

We first collect the geometric facts that let us choose the right Killing
field, differentiate the eigenfunction along it, and control the resulting
boundary trace.

\begin{lemma}\label{lem:translation}
Let \(p\in\HH^2\) and let \(v\in T_p\HH^2\) be a unit vector. There is a
Killing field \(K_v\) generating translation along the geodesic through \(p\)
with initial velocity \(v\), normalized so that
\[
  K_v(p)=v,
  \qquad
  \nabla K_v|_p=0.
\]
\end{lemma}

\begin{proof}
Let \(\gamma\) be the geodesic with
\(\gamma(0)=p\) and \(\gamma'(0)=v\), and let \(K_v\) generate the
one-parameter group of hyperbolic translations along \(\gamma\), normalized
to have unit length on \(\gamma\). Thus \(K_v(p)=v\) and
\(\nabla_{K_v}K_v=0\) on the axis.

On an oriented surface, \(\nabla K_v\) is skew-symmetric, so at \(p\) it has
the form
\[
  \nabla_XK_v=\omega JX
\]
for a scalar \(\omega\), where \(J\) denotes rotation by \(\pi/2\).
Since
\(\nabla_{K_v}K_v(p)=\omega Jv=0\), one has \(\omega=0\), hence
\(\nabla K_v|_p=0\).
\end{proof}

The next fact keeps the Killing derivative in the same eigenspace.

\begin{lemma}\label{lem:commutator}
If \(K\) is a Killing field on \(\HH^2\), then
\[
  \Delta(Kf)=K(\Delta f)
\]
for every smooth function \(f\). Consequently, if \(f\) solves
\((\Delta+\lambda)f=0\), then \(Kf\) solves the same equation.
\end{lemma}

\begin{proof}
The flow of \(K\) consists of isometries. The Laplace--Beltrami operator is
invariant under isometries, and differentiating this invariance at time zero
gives the commutator identity.
\end{proof}

The shift introduced in Section 3 vanishes on the boundary. Its boundary
behavior is therefore inherited from \(K(\psi_1)\). The following result,
stated in the form used here, is the only place where horoconvexity enters
the proof.

\begin{proposition}[Grossi--Provenzano \cite{GrossiProvenzano}]\label{prop:GP}
Let \(\Om\subset\HH^2\) be bounded, smooth, and horoconvex. Let
\(p\in\Om\), let \(v\in T_p\HH^2\) be a unit vector, and let \(K\) be the
translation Killing field whose axis is the geodesic through \(p\) tangent
to \(v\). If \(\psi\in C^1(\overline\Om)\) satisfies
\(\psi>0\) in \(\Om\), \(\psi=0\) on \(\partial\Om\), and
\(\nabla\psi\neq0\) on \(\partial\Om\), then \(K(\psi)\) has exactly two
zeros on \(\partial\Om\).
\end{proposition}

\begin{proof}
This is \cite[Proposition 3.15(iii)]{GrossiProvenzano}. The cited proposition
allows arbitrary \(v\in T_p\HH^2\). In particular, it imposes no
orthogonality condition between \(v\) and \(\nabla\psi(p)\).
\end{proof}

\begin{remark}
In the Poincar\'e disk with \(p\) at the origin, the integral curves of \(K\)
are hypercycles. A boundary zero of \(K(\psi)\) is a tangency between
\(\partial\Om\) and one of these hypercycles. Grossi and Provenzano use
horoconvexity to show that there is exactly one tangency on each side of the
axis.
\end{remark}

\section{From degeneracy to a singular zero}

Set
\(
  u=-\log\psi_1
\).
Equation \eqref{eq:eigenfunction} gives
\begin{equation}\label{eq:u}
  \Delta u=\lambda_1+|\nabla u|^2.
\end{equation}

The next lemma is the bridge from the Hessian question to a linear nodal
problem. The calculation is short; the choice of the shift is the main
point.

\begin{lemma}[Shifted Killing derivative]\label{lem:shift}
Suppose that \(\Hess u(p)\) is degenerate at \(p\in\Om\), and let
\(v\in T_p\HH^2\) be a unit vector in its kernel. Let \(K=K_v\) be the
translation field from Lemma \ref{lem:translation}, and define
\begin{equation}\label{eq:shift}
  z:=K(\psi_1)+c_0\psi_1,
  \qquad
  c_0=\langle v,\nabla u(p)\rangle.
\end{equation}
Then
\[
  (\Delta+\lambda_1)z=0,\qquad
  z(p)=0,\qquad
  \nabla z(p)=0,
\]
and
\begin{equation}
\label{eq:boundary-Z}
  z|_{\partial\Om}=K(\psi_1)|_{\partial\Om}.
\end{equation}
\end{lemma}

\begin{proof}
Lemma \ref{lem:commutator} and \eqref{eq:eigenfunction} show that both
\(K(\psi_1)\) and \(\psi_1\) solve the equation
\((\Delta+\lambda_1)f=0\). Since \(\psi_1=0\) on \(\partial\Om\), the shift
does not change the boundary trace.

Using \(\psi_1=e^{-u}\), we have \(
  K(\psi_1)=-\psi_1K(u),\) \(c_0=K(u)(p), \) therefore
\[
  z=-\psi_1\big[K(u)-K(u)(p)\big].
\]
This proves \(z(p)=0\). If \( h=K(u)-K(u)(p), \)
then
\[
  \nabla h
  =\Hess u(K,\cdot)
   +\langle\nabla_{\,\cdot\,}K,\nabla u\rangle.
\]
At \(p\), Lemma \ref{lem:translation} gives
\(K(p)=v\) and \(\nabla K|_p=0\). The kernel assumption gives
\(\Hess u(v,\cdot)=0\). Hence \(\nabla h(p)=0\), and
\[
  \nabla z(p)
  =-\psi_1(p)\nabla h(p)-h(p)\nabla\psi_1(p)
  =0. \qedhere
\]
\end{proof}

\begin{remark}\label{rem:new-step}
Without the shift, the kernel condition gives
\(\nabla(K (\psi_1))(p)=0\), but it does not give \(K( \psi_1)(p)=0\), since the kernel
direction \(v\) need not be orthogonal to \(\nabla \psi_1(p)\). Choosing a
different Killing field to force this orthogonality could leave the
translation class covered by Proposition \ref{prop:GP}. The term
\(c_0\psi_1\) removes the unwanted value while preserving both the equation
and the boundary trace.
\end{remark}

\begin{remark}
\label{rem:Z-not-zero}
Proposition \ref{prop:GP} and \eqref{eq:boundary-Z} imply that $z$ has only two zeros on the boundary, so it does not vanish identically.
\end{remark}

\section{Two boundary zeros and singular nodal points}

We now show in Proposition~\ref{prop:nodal} that the boundary-zero count rules out the singular zero
constructed above. Its proof is implicitly contained in the proof of \cite[Proposition 3.5]{GrossiProvenzano}. The key elements are the domain monotonicity of the eigenvalue and the local description of the nodal line  on surfaces for the solution of \eqref{eq:Z} \cite{Cheng}. We give a more detailed proof for our case.

\begin{proposition}\label{prop:nodal}
Let \(\Om\) be a bounded smooth simply connected domain in a smooth
Riemannian surface. Let \(\psi>0\) be a first Dirichlet eigenfunction:
\begin{equation}
  (\Delta+\lambda_1(\Om))\psi=0
  \quad\text{in }\Om,
  \qquad
  \psi=0
  \quad\text{on }\partial\Om. \label{eq:eigen function psi}
\end{equation}
Suppose that \(z\in C^2(\Om)\cap C^1(\overline\Om)\), \(z\not\equiv0\),
solves
\begin{equation}
  (\Delta+\lambda_1(\Om))z=0
  \quad\text{in }\Om  \label{eq:Z}
\end{equation}
and has exactly two zeros on \(\partial\Om\). Then \(z\) has no singular
zero in \(\Om\).
\end{proposition}

\begin{proof} By a nodal domain we mean a connected component of the set where
$z$ does not vanish. 

Suppose, for contradiction, that \(z\) has a singular zero
at a point \(p\in\Om\), that is, \(z(p)=0\) and \(\nabla z(p)=0\).
We divide the proof into three steps.

\smallskip
\noindent\emph{Step 1: every nodal domain meets an open boundary arc.}
Here, a nodal domain meets an open boundary arc if its closure contains that
arc. Suppose that a nodal domain \(\Omega_+\) does not do so. Thus, it is a proper subdomain of $\Omega$. By strict domain monotonicity of the principal Dirichlet eigenvalue, any proper subdomain \(\Omega_+ \subsetneq \Omega\) must satisfy 
\[
\lambda_1(\Omega_+) > \lambda_1(\Omega).
\]
Since the nodal domain \(\Omega_+\) is strictly contained in \(\Omega\) and supports a positive eigenfunction corresponding to the same eigenvalue, we obtain \(\lambda_1(\Omega_+) = \lambda_1(\Omega)\), a contradiction.

\smallskip
\noindent\emph{Step 2: $z$ has exactly two nodal domains.} 
By hypothesis, the boundary trace of \(z\) has exactly two zeros on \(\partial\Omega \cong S^1\). Green's identity with the positive first eigenfunction gives
\[
0=\int_{\partial\Om}z\,\partial_\nu\psi\,ds.
\]
By the Hopf lemma, $\partial_\nu\psi<0$. Thus the two boundary zeros force a partition of the boundary curve into exactly one positive open arc \(\Gamma^+\) where $z>0$ and one negative open arc \(\Gamma^-\) where $z<0$.

Each point of \(\Gamma^+\) has a boundary collar contained in a single
positive nodal domain. The selected nodal domain is locally fixed along
the connected arc, hence fixed. Step 1 says that every positive nodal
domain meets \(\Gamma^+\), so $\{z > 0\}$ is connected. The same argument shows
that $\{z < 0\}$ is connected. Thus $z$ has exactly two nodal domains. 

\smallskip
\noindent\emph{Step 3: global Jordan separation and contradiction.}
By the local structure of nodal sets (see \cite{Cheng}), the vanishing order of \(z\) at \(p\) is at least \(2\), and a small neighborhood of \(p\) is partitioned into \(2k \ge 4\) alternating positive and negative nodal sectors meeting at \(p\). 

Choose four such alternating sectors. Since the positive nodal domain is connected, short arcs from $p$ into the selected positive sectors can be joined by a simple arc in $\{z > 0\}$. After trimming intersections, their union is a
Jordan curve $L$ with $L\setminus \{p\} \subset \{z > 0\}$.

Since $L\subset \Omega$ and $\Omega$ is simply connected, the Jordan–Schoenflies theorem shows that $\Omega\setminus L$ has two components. 
Near $p$, the two selected negative sectors lie on opposite local sides of $L$ and therefore in different components of $\Omega\setminus L$. But $\{z < 0\}$ is connected, is disjoint from $L$, and
contains points in both sectors. This is a contradiction.
\end{proof}



\section{Log-concavity of the First Eigenfunction}

\begin{proof}[Proof of Theorem \ref{thm:main}]
We now carry out the reduction described in the introduction.
Let \(u=-\log\psi_1\). We first show that \(\Hess u\) is nonsingular at every
point of \(\Om\).

Suppose instead that \(\Hess u(p)\) is degenerate. Since it is a symmetric
endomorphism of the two-dimensional space \(T_p\HH^2\), it has a unit kernel
vector \(v\). Construct \(z\) by Lemma \ref{lem:shift}. Hopf's lemma gives
\(\nabla\psi_1\neq0\) on \(\partial\Om\), so Proposition \ref{prop:GP}
applies. Together with the boundary identity in Lemma \ref{lem:shift}, it
shows that \(z\) has exactly two boundary zeros. In particular
\(z\not\equiv0\).

A horoconvex domain in \(\HH^2\) is geodesically convex, hence simply
connected. Proposition \ref{prop:nodal} therefore applies to \(z\), but Lemma
\ref{lem:shift} gives
\[
  z(p)=0,
  \qquad
  \nabla z(p)=0.
\]
This contradiction proves
\[
  \det\Hess u\neq0
  \qquad\text{throughout }\Om.
\]

The determinant has constant sign because \(\Om\) is connected. Let \(x_0\)
be a maximum point of \(\psi_1\). Then \(u\) has a local minimum at \(x_0\), so
\(\Hess u(x_0)\geq0\). Nonsingularity makes this inequality strict, and hence
\(\det\Hess u>0\) throughout \(\Om\). Finally, \eqref{eq:u} gives
\[
  \tr\Hess u
  =\Delta u
  =\lambda_1+|\nabla u|^2
  >0.
\]
A symmetric two-by-two matrix with positive determinant and positive trace is
positive definite. Thus \(\Hess u>0\) on \(\Om\).
\end{proof}





\begin{remark}[Other space forms]\label{rem:other-space-forms}
The same shifted proof gives alternate proofs in the other two-dimensional
space forms under the remaining clauses of the Grossi--Provenzano
boundary-zero theorem. More precisely,
\[
  \Hess_{\mathbb R^2}(-\log\psi_1)>0
\]
on every bounded smooth convex planar domain with positive boundary
curvature, and
\[
  \Hess_{\mathbb S^2}(-\log\psi_1)>0
\]
on every bounded smooth geodesically convex spherical domain with positive
boundary curvature and diameter less than \(\pi/2\). These follow from
\cite[Proposition 3.15(i),(ii)]{GrossiProvenzano}. For the required field in either space form, the proof of Lemma
\ref{lem:translation} again gives \(K(p)=v\) and \(\nabla K|_p=0\). Strict spherical log-concavity is already known under broader strict
convexity hypotheses through Lee and Wang \cite{LeeWang}. 
The present
argument is an alternate proof under the additional condition of a diameter less than \(\pi/2\).
\end{remark}

\section{General Killing fields}\label{sec:killing2}

The proof of Theorem \ref{thm:level} uses Killing fields whose covariant
derivative does not vanish at the base point. This section records their
structure. Fix an orientation of \(\HH^2\) and let \(J\) denote rotation by
\(\pi/2\) in the tangent planes, so \(J\) is parallel, skew
(\(\langle JX,Y\rangle=-\langle X,JY\rangle\)), and \(J^2=-\mathrm{Id}\).

\begin{lemma}\label{lem:killing-structure}
Let \(K\) be a Killing field on \(\HH^2\). Then there is a smooth function
\(\omega\) with
\begin{equation}\label{eq:nablaK}
  \nabla_XK=\omega\,JX\qquad\text{for all }X,
\end{equation}
and
\begin{equation}\label{eq:domega}
  d\omega=-\langle JK,\cdot\rangle .
\end{equation}
The quantity
\begin{equation}\label{eq:invariant}
  I_K:=|K|^2-\omega^2
\end{equation}
is constant on \(\HH^2\). Moreover, for every \(p\in\HH^2\) the evaluation
map \(K\mapsto(K(p),\omega(p))\) is a linear bijection from the space of
Killing fields onto \(T_p\HH^2\times\mathbb R\).
\end{lemma}

\begin{proof}
For a Killing field, \(\nabla K\) is skew-symmetric; on an oriented surface
the skew-symmetric endomorphisms of a tangent plane are the multiples of
\(J\), which gives \eqref{eq:nablaK}.

Killing fields satisfy the second-order identity
\begin{equation}
    (\nabla_X \nabla K)Y =\nabla_X \nabla_Y K -\nabla_{\nabla_XY} K =  -R(K,X)Y,
    \end{equation}
see e.g. \cite[Proposition 8.1.3]{Petersen}.
    Since
\(J\) is parallel, the left side is \(d\omega(X)\,JY\). For curvature
\(-1\), \(R(X,Y)Z=\langle X,Z\rangle Y-\langle Y,Z\rangle X\), so
\[
  d\omega(X)\,JY=\langle X,Y\rangle K-\langle K,Y\rangle X .
\]
Pairing with \(JY\) for a unit vector \(Y\) and expanding
\(JK=-\langle K,JY\rangle Y+\langle K,Y\rangle JY\) gives
\eqref{eq:domega}.

From \eqref{eq:nablaK} and \eqref{eq:domega},
\begin{align*}
  d|K|^2(X)&=2\langle\nabla_XK,K\rangle=2\omega\langle JX,K\rangle\\
  &=-2\omega\langle X,JK\rangle
  =2\omega\,d\omega(X)=d(\omega^2)(X),
\end{align*}
so \(I_K\) is constant.

The map \(K\mapsto(K(p),\omega(p))\) is linear between three-dimensional
spaces, so it suffices to show injectivity. If \(K(p)=0\) and
\(\omega(p)=0\), then \eqref{eq:nablaK} and \eqref{eq:domega} form a linear
first-order system for \((K,\omega)\) along any curve from \(p\), with zero
initial data; hence \(K\equiv0\).
\end{proof}

\begin{remark}
For the translation field \(K_v\) of Lemma \ref{lem:translation} one has
\(\omega(p)=0\) and \(|K_v(p)|=1\), so \(I_{K_v}=1\).
\end{remark}

\begin{lemma}\label{lem:classification}
Let \(K\not\equiv0\) be a Killing field on \(\HH^2\) with \(I_K>0\). Then
\(K\) generates the one-parameter group of hyperbolic translations along a
geodesic \(\gamma\), its axis, and the integral curves of \(K\) are
\(\gamma\) and its equidistant curves. Moreover \(|K|^2\ge I_K>0\), so
\(K\) has no zeros.
\end{lemma}

\begin{proof}
\(K\) generates a one-parameter group of orientation-preserving isometries
of \(\HH^2\), which is elliptic, parabolic, or hyperbolic.

If the group is elliptic with fixed point \(q\in\HH^2\), then \(K(q)=0\),
so \(I_K=-\omega(q)^2\le0\).

If the group is parabolic, then after conjugation by an isometry and
scaling we may take \(K=\partial_x\) in the upper half-plane model
\(g=y^{-2}(dx^2+dy^2)\). Then
\(\nabla_{\partial_x}\partial_x=y^{-1}\partial_y=y^{-1}J\partial_x\), so
\(\omega=1/y\), while \(|K|^2=y^{-2}\); hence \(I_K=0\). Conjugation
preserves \(I_K\) and scaling preserves its sign, so \(I_K=0\) for every
parabolic field.

Hence \(I_K>0\) forces the hyperbolic case: the group translates along a
geodesic \(\gamma\), and its orbits are \(\gamma\) and the curves at fixed
distance from \(\gamma\). Finally \(|K|^2=I_K+\omega^2\ge I_K\).
\end{proof}

Lemma \ref{lem:shift-general} extends the shifted derivative of Lemma \ref{lem:shift} to the setting when $\omega(p) \neq 0$;
the only change in the computation is the term produced by
\(\nabla K\neq0\). When \(\omega(p)=0\) and \(K(p)=v\) lies in the kernel of
\(\Hess u(p)\), we recover Lemma \ref{lem:shift}.

\begin{lemma}\label{lem:shift-general}
Let \(K\) be a Killing field, let \(p\in\Om\), and define
\begin{equation}\label{eq:shift-general}
  z:=K(\psi_1)+c_0\psi_1,
  \qquad
  c_0=K(u)(p).
\end{equation}
Then \((\Delta+\lambda_1)z=0\),
\(z|_{\partial\Om}=K(\psi_1)|_{\partial\Om}\), \(z(p)=0\), and
\[
  \nabla z(p)
  =-\psi_1(p)\bigl(\Hess u|_p\,K(p)-\omega(p)\,J\nabla u(p)\bigr).
\]
In particular, \(z\) has a singular zero at \(p\) if and only if
\(\Hess u|_p\,K(p)=\omega(p)\,J\nabla u(p)\).
\end{lemma}

\begin{proof}
Lemma \ref{lem:commutator} and \eqref{eq:eigenfunction} show that
\(K(\psi_1)\) and \(\psi_1\) both solve \((\Delta+\lambda_1)f=0\), and the
shift vanishes on \(\partial\Om\). From \(\psi_1=e^{-u}\) we get
\(K(\psi_1)=-\psi_1K(u)\) and \(c_0=-K(\psi_1)(p)/\psi_1(p)\), so
\[
  z=-\psi_1\bigl[K(u)-K(u)(p)\bigr]
\]
and \(z(p)=0\). Differentiating \(K(u)=\langle K,\nabla u\rangle\) with
\eqref{eq:nablaK},
\begin{align*}
  X\bigl(K(u)\bigr)
  &=\Hess u(K,X)+\langle\nabla_XK,\nabla u\rangle
  =\Hess u(K,X)+\omega\langle JX,\nabla u\rangle\\
  &=\bigl\langle \Hess u\,K-\omega J\nabla u,\;X\bigr\rangle,
\end{align*}
using skewness of \(J\) in the last step. At \(p\) the factor
\(K(u)-K(u)(p)\) vanishes, so
\(\nabla z(p)=-\psi_1(p)\nabla\bigl(K(u)\bigr)(p)\), which is the displayed
formula. 
\end{proof}

\section{A boundary-zero count for arbitrary axes}\label{sec:count}

This section proves the count that replaces Proposition \ref{prop:GP} when
the axis of the translation field does not pass through \(\Om\). We use two
standard Hessian formulas in \(\HH^2\): for a Busemann function \(b\)
(so \(|\nabla b|=1\), the sublevel sets \(\{b\le c\}\) are horoballs, and
the level curves are horocycles),
\begin{equation}\label{eq:hessb}
  \Hess b=g-db\otimes db,
\end{equation}
and for the signed distance \(s\) to a geodesic \(\gamma\),
\begin{equation}\label{eq:hesss}
  \Hess s=\tanh(s)\,\bigl(g-ds\otimes ds\bigr).
\end{equation}
Both follow from the Riccati comparison along the normal geodesics: the
level curves of \(b\) are horocycles of curvature \(1\), and the level
curves of \(s\) are the equidistant curves of \(\gamma\), of curvature
\(|\tanh s|\) with curvature vector pointing toward \(\gamma\).

\begin{lemma}\label{lem:two-points}
Let \(\Om\subset\HH^2\) be a bounded smooth horoconvex domain and let
\(h\subset\HH^2\) be a complete embedded curve of constant geodesic
curvature \(\kappa_0\in[0,1)\), that is, a geodesic or an equidistant
curve. Then \(\partial\Om\cap h\) contains at most two points.
\end{lemma}

\begin{proof}
Parametrize \(h\) by arclength, \(h:\mathbb R\to\HH^2\), with covariant
acceleration \(|\nabla_{h'}h'|=\kappa_0\). Let \(b\) be a Busemann function
whose horoball \(\{b\le c\}\) contains \(\Om\), and set
\(f=b\circ h\). At a critical point \(t\) of \(f\) we have
\(\langle\nabla b,h'\rangle=0\), so by \eqref{eq:hessb}
\[
  f''(t)=\Hess b(h',h')+\langle\nabla b,\nabla_{h'}h'\rangle
  =1+\langle\nabla b,\nabla_{h'}h'\rangle\ge1-\kappa_0>0 .
\]
Every critical point of \(f\) is a strict local minimum. Consequently
every sublevel set \(\{f\le c\}\) is an interval: otherwise \(f\) would
attain an interior local maximum value \(>c\) on the interval between two
sublevel components, and each maximum point (or plateau point) would be a
critical point that is not a strict local minimum.

By Definition \ref{def:horoconvex}, \(\overline\Om\) is an intersection of
horoballs, so \(\overline\Om\cap h\) is an intersection of
intervals, hence an interval \(I\subset h\) (possibly empty or a point).

Let \(x\in\partial\Om\cap h\) and suppose \(x\) is an interior point of
\(I\). As recalled after Definition \ref{def:horoconvex}, there is also a
supporting horoball
at \(x\): a Busemann function \(b\) with \(\overline\Om\subset\{b\le c\}\)
and \(b(x)=c\). Then \(f=b\circ h\) satisfies \(f\le c\) on \(I\) with
equality at the interior point \(x\), so \(x\) is a local maximum of \(f\),
hence a critical point that is not a strict local minimum, a contradiction.
Therefore \(\partial\Om\cap h\subset\partial I\), which contains at most
two points.
\end{proof}

\begin{proposition}\label{prop:count}
Let \(\Om\subset\HH^2\) be a bounded smooth horoconvex domain and let
\(\psi\in C^1(\overline\Om)\) satisfy \(\psi>0\) in \(\Om\), \(\psi=0\) and
\(\nabla\psi\neq0\) on \(\partial\Om\). Let \(K\not\equiv0\) be a Killing
field with \(I_K>0\). Then \(K(\psi)\) has exactly two zeros on
\(\partial\Om\).
\end{proposition}

\begin{proof}
By Lemma \ref{lem:classification}, \(K\) is a translation field with some
axis \(\gamma\), it has no zeros, and its integral curves are the level
curves of the signed distance \(s\) to \(\gamma\). On \(\partial\Om\) the
gradient \(\nabla\psi\) is a nonzero normal vector, so for
\(x\in\partial\Om\),
\begin{align*}
  K(\psi)(x)=0
  &\iff K(x)\in T_x\partial\Om\\
  &\iff \text{the level curve of } s \text{ through } x
       \text{ is tangent to } \partial\Om \text{ at } x\\
  &\iff x \text{ is a critical point of } s|_{\partial\Om}.
\end{align*}

Parametrize \(\partial\Om\) by arclength \(\sigma\mapsto\alpha(\sigma)\)
with inward normal \(N\) and geodesic curvature \(k_g\ge1\), so
\(\nabla_{\alpha'}\alpha'=k_gN\). Let \(x=\alpha(0)\) be a critical point
of \(f:=s\circ\alpha\), lying on the level curve \(\{s=s_0\}\). By
\eqref{eq:hesss} and \(\langle\nabla s,\alpha'\rangle=0\) at \(\sigma=0\),
\[
  f''(0)=\Hess s(\alpha',\alpha')+\langle\nabla s,k_gN\rangle
  =\tanh(s_0)\pm k_g .
\]
Since \(k_g\ge1>|\tanh(s_0)|\), \(f''(0)\neq0\): every critical point of
\(f\) is a nondegenerate strict local extremum. In particular the critical
points are isolated, hence finitely many, and along \(\partial\Om\) they
alternate between local maxima and local minima, say \(k\) of each,
\(k\ge1\). There is at least one of each because \(f\) attains its maximum
and minimum.

Suppose \(k\ge2\). Then \(f\) has two strict local maxima \(x_1,x_2\) with
values \(s_1\ge s_2\). Each of the two open arcs of \(\partial\Om\) between
\(x_1\) and \(x_2\) contains points with value \(<s_2\), immediately beside
the strict maximum \(x_2\), so the minimum of \(f\) over each open arc is
\(<s_2\). Choose a regular value \(c\) of \(f\) with
\(\max(\text{arc minima})<c<s_2\). On each arc, \(f\) passes from a value
\(\ge s_2>c\) at the endpoints to a value \(<c\) inside, so \(f=c\) at
least twice per arc: the level \(\{f=c\}\) contains at least four points.
All four lie on the single level curve \(\{s=c\}\), which is a geodesic or
an equidistant curve of curvature \(|\tanh c|<1\). This contradicts
Lemma \ref{lem:two-points}. 

Hence \(k=1\): the function \(s|_{\partial\Om}\) has exactly one maximum
and one minimum, and \(K(\psi)\) has exactly two zeros on \(\partial\Om\).
\end{proof}

\begin{remark}
For axes through a point of \(\Om\), Proposition \ref{prop:count} is
Proposition \ref{prop:GP}; the one-tangency-per-side argument of
\cite{GrossiProvenzano} uses that the axis crosses the domain. The proof
above makes no assumption on the position of the axis, which is what
Theorem \ref{thm:level} requires. Note that for a  hyperbolic Killing field \(K\) i.e.\ \(I_K>0\), with axis
\(\gamma\), for every \(p\in\HH^2\),
\begin{equation}\label{eq:invariant-dist}
  \dist(p,\gamma)=\artanh\!\Bigl(\frac{|\omega(p)|}{|K(p)|}\Bigr).
\end{equation}
Hence,  for the axis of Killing field in \eqref{eqn:Killing with 1} of
Section \ref{sec:level}, the distance is \(\operatorname{artanh}(1/|K(p)|)\)
from \(p\) and can lie outside \(\overline\Om\).
\end{remark}

\section{Horoconvexity of the level sets}\label{sec:level}

\begin{proposition}\label{prop:star}
Let \(\Om\subset\HH^2\) be a bounded smooth horoconvex domain and let
\(u=-\log\psi_1\). Then \eqref{eq:star} holds:
\(|(\Hess u)^{-1}J\nabla u|\le1\) on \(\Om\).
\end{proposition}

\begin{proof}
\(\Hess u>0\) by Theorem \ref{thm:main}, so the left side is defined.
Suppose \(m(p):=|(\Hess u)^{-1}J\nabla u|(p)>1\) at some \(p\in\Om\); note
\(\nabla u(p)\neq0\).

By Lemma \ref{lem:killing-structure} there is a unique Killing field \(K\)
with
\begin{equation}
  K(p)=(\Hess u)^{-1}J\nabla u(p),
  \qquad
  \omega(p)=1.  \label{eqn:Killing with 1}
\end{equation}
Its invariant is \(I_K=|K(p)|^2-1=m(p)^2-1>0\), so \(K\) is a translation
field by Lemma \ref{lem:classification}.

Let \(z=K(\psi_1)+K(u)(p)\psi_1\) as in \eqref{eq:shift-general}. By
construction \(\Hess u|_p K(p)=J\nabla u(p)=\omega(p)J\nabla u(p)\), so
Lemma \ref{lem:shift-general} gives \(z(p)=0\) and \(\nabla z(p)=0\).

By Hopf's lemma \(\nabla\psi_1\neq0\) on \(\partial\Om\), so
Proposition \ref{prop:count} applies to \(K\) and \(\psi_1\): the boundary
trace \(z|_{\partial\Om}=K(\psi_1)|_{\partial\Om}\) has exactly two zeros;
in particular \(z\not\equiv0\). A horoconvex domain is geodesically convex,
hence simply connected, so Proposition \ref{prop:nodal} applies and \(z\)
has no singular zero in \(\Om\). This contradicts \(z(p)=0\),
\(\nabla z(p)=0\). Hence \(m\le1\) on \(\Om\).
\end{proof}

\begin{proof}[Proof of Theorem \ref{thm:level}]
By Theorem \ref{thm:main} the function \(u\) is strictly convex, so it has
at most one critical point; since \(\psi_1\) attains its maximum, \(u\) has
exactly one, its minimum \(x_0\). Every value \(t\in(0,\max\psi_1)\) is
regular, so \(\{\psi_1=t\}=\partial\{\psi_1>t\}\) is a smooth compact
curve. The superlevel set \(\{\psi_1>t\}\) is geodesically convex because
\(\log\psi_1\) is concave along geodesics, so its boundary is a single
smooth closed curve.

Fix a point \(x\) on this level curve and write \(G=|\nabla u|(x)>0\). Let
\(N=-\nabla u/G\) be the unit normal pointing into the superlevel set,
parametrize the level curve by arclength \(\alpha\) with unit tangent
\(\tau\), and define the inward geodesic curvature \(\kappa\) by
\(\nabla_\tau\tau=\kappa N\). Differentiating
\(u\circ\alpha=\mathrm{const}\) twice,
\[
  0=\Hess u(\tau,\tau)+\langle\nabla u,\nabla_\tau\tau\rangle
   =\Hess u(\tau,\tau)-\kappa G,
\]
so \(\kappa=\Hess u(\tau,\tau)/G\).

Set \(n=\nabla u/G\) and orient \(\alpha\) so that \(\tau=Jn\); this does
not change \(\kappa\). Write \(a=\Hess u(\tau,\tau)\),
\(b=\Hess u(\tau,n)\), \(c=\Hess u(n,n)\), so \(a,c>0\) and
\(D:=ac-b^2>0\). Since \(J\nabla u=G\tau\),
\[
  (\Hess u)^{-1}J\nabla u=\frac{G}{D}\,(c\,\tau-b\,n),
  \qquad
  \bigl|(\Hess u)^{-1}J\nabla u\bigr|=\frac{G\sqrt{b^2+c^2}}{D}.
\]
Proposition \ref{prop:star} gives \(G\sqrt{b^2+c^2}\le D\). Then
\[
  ac=D+b^2\ge G\sqrt{b^2+c^2}+b^2\ge Gc+b^2,
\]
so
\begin{equation}\label{eq:kappabound}
  \kappa=\frac aG\;\ge\;1+\frac{b^2}{Gc}\;\ge\;1 .
\end{equation}
Every level curve \(\{\psi_1=t\}\), \(t\in(0,\max\psi_1)\), therefore has
geodesic curvature at least \(1\), and the superlevel sets are bounded
smooth domains with \(k_g\ge1\), hence horoconvex by the curvature
reformulation recalled after Definition \ref{def:horoconvex}. For \(t\le0\) the
superlevel set is \(\Om\) itself; for \(t\ge\max\psi_1\) it is empty or the
single point \(x_0\).
\end{proof}

\begin{remark}\label{rem:sharp}
Inequality \eqref{eq:star} is sharp. If \(u=f(b)\) for a Busemann function
\(b\) and a function \(f\) with \(f'>0\), then \eqref{eq:hessb} gives
\(\Hess u=f''\,db\otimes db+f'\,(g-db\otimes db)\), so
\(\Hess u\,(J\nabla b)=f'J\nabla b\) and
\((\Hess u)^{-1}J\nabla u=J\nabla b\) has norm \(1\): these functions
saturate \eqref{eq:star} identically, for every \(f''\). At a point where
equality holds in \eqref{eq:star}, the field constructed in the proof of
Proposition \ref{prop:star} has \(I_K=0\) and is parabolic; its integral
curves are horocycles, for which both comparison arguments of
Section \ref{sec:count} degenerate.
\end{remark}

\begin{remark}\label{rem:balls}
For the radial eigenfunction of a geodesic ball, in the frame
\((\tau,n)\) at distance \(t\) from the center, \(a=\coth(t)\,u'\),
\(c=u''\), \(b=0\), and \(G=u'\), so
\(|(\Hess u)^{-1}J\nabla u|=G/a=\tanh t<1\) for every radius. The
large-ball failure of super log-concavity \cite{WeiXiao} lives in the
normal direction and is invisible to \eqref{eq:star}; this is why no
diameter restriction appears in Theorem \ref{thm:level}.
\end{remark}

\begin{remark}\label{rem:hierarchy}
Between super log-concavity and horoconvexity of the level sets,
\eqref{eq:star} is intermediate. Super log-concavity written for
\(u=-\log\psi_1\) is the tensor inequality
\(\Hess u\ge|\nabla u|\,g\); it bounds every eigenvalue of \(\Hess u\)
below by \(G=|\nabla u|\), so
\[
  \bigl|(\Hess u)^{-1}J\nabla u\bigr|\le\frac{|J\nabla u|}{G}=1 .
\]
Thus super log-concavity implies \eqref{eq:star}, and \eqref{eq:star}
implies \(\kappa\ge1\) by the proof of Theorem \ref{thm:level}. Neither
implication reverses: \eqref{eq:star} imposes no lower bound on
\(\Hess u(n,n)\), the normal bound that fails for large balls in higher
dimension \cite{WeiXiao}, and \(\kappa\ge1\) carries no bound on
\(\Hess u(\tau,n)\), which \eqref{eq:star} controls through
\eqref{eq:kappabound}.
\end{remark}

\section{Research process}
The shifted construction was found after an extended computer-assisted search
that included symbolic calculations, numerical counterexample searches, and
language-model-assisted proof exploration and review. The extension to
general Killing fields and Theorem \ref{thm:level} were found the same way;
the identities of Lemma \ref{lem:killing-structure} and the inequality
\eqref{eq:star} were checked symbolically and numerically before the proofs
were written. None of those calculations or model outputs is used as
evidence in the proofs above.

\section*{Acknowledgements}
We thank Prof. Luigi Provenzano for his interest in this paper and for
suggesting the problem of horoconvexity of the superlevel sets.

\small\raggedright

\end{document}